\documentclass[11pt]{article}

\usepackage{amssymb,amsfonts,amsmath, cite}
\usepackage{enumerate}

\usepackage[english]{babel}
\usepackage[cp1250]{inputenc}

\usepackage{algorithm2e}

\usepackage{longtable}

\usepackage{graphicx}
\usepackage{index}
\usepackage{fancyhdr}
\usepackage{lhelp}
\usepackage{optparams}
\usepackage{psfrag}
\usepackage{forloop}
\usepackage{setspace}
\usepackage{tikz}
\usepackage{subcaption}
\usepackage{pgffor}
\usepackage{xifthen}

\definecolor{lgray}{gray}{0.75}

\usetikzlibrary{decorations.pathreplacing,automata,calc,positioning}

\newcommand{\pch}{\chi_{\rho}}
\newcommand{\diam}{{\rm diam}}

\newcommand{\qed}{\hfill $\square$ \bigskip}

\newcommand{\mptt}[1]{}

\newtheorem{theorem}{Theorem} %ce pise na koncu Chapter, bo potem vse ?tevil?eno 1.15 in podobno%

\newtheorem{lemma}[theorem]{Lemma}

\newtheorem{question}[theorem]{Question}

\begin{document}

\title{\bf An infinite family of subcubic graphs with unbounded packing chromatic number}

\author{
Bo\v{s}tjan Bre\v{s}ar $^{a,b}$  \and Jasmina Ferme $^{c,a}$ 
 }

\date{}

\maketitle

\begin{center}
$^a$ Faculty of Natural Sciences and Mathematics, University of Maribor, Slovenia\\
\medskip

$^b$ Institute of Mathematics, Physics and Mechanics, Ljubljana, Slovenia\\
\medskip

$^c$ Faculty of Education, University of Maribor, Slovenia\\

\end{center}

\begin{abstract}
Recently, Balogh, Kostochka and Liu in [Packing chromatic number of cubic graphs, Discrete Math.~341 (2018) 474--483] answered in negative the question that was posed in several earlier papers whether the packing chromatic number is bounded in the class of graphs with maximum degree $3$. In this note, we present an explicit infinite family of subcubic graphs with unbounded packing chromatic number. 
\end{abstract}

\noindent {\bf Key words:} packing, coloring, packing coloring, diameter, subcubic graphs.

\medskip\noindent
{\bf AMS Subj.\ Class:} 05C70, 05C15, 05C12

\bigskip

Given a graph $G$, the {\em distance} between two vertices $u$ and $v$ in $G$, denoted by $d_G(u,v)$, is the length of a shortest $u,v$-path (we often drop the subscript if the graph $G$ is clear from context). The maximum of $\{d_G(x,y)\,|\, x,y \in V(G)\}$ is called the  {\em diameter} of $G$ and denoted by $\diam(G)$.
An {\em $i$-packing} in $G$, where $i$ is a positive integer, is a subset $W$ of the vertex set of $G$ such that the distance between any two distinct vertices from $W$ is greater than $i$. This concept generalizes the notion of an independent set, which is equivalent to a $1$-packing. The {\em packing chromatic number} of $G$ is the smallest integer $k$ such that the vertex set of $G$ can be partitioned into sets $V_1,\ldots, V_k$, where $V_i$ is an $i$-packing for each $i\in \{1,\ldots, k\}$. This invariant is well defined in any graph $G$ and is denoted by $\pch(G)$. The corresponding mapping $c:V(G)\longrightarrow \{1, \ldots, k\}$ having the property that $c(u)=c(v)=i$ implies $d_G(u, v) > i$ is called a {\em $k$-packing coloring}. 

The concept of packing chromatic number of a graph was introduced a decade ago under the name broadcast chromatic number~\cite{goddard-2008}, and the current name was given in~\cite{bkr-2007}. A number of authors have studied this invariant, cf. a selection of recent papers~\cite{argiroffo-2014,barnaby-2017, bkr-2016, bkrw-2017a,bkrw-2017b,ekstein-2014,finbow-2010,gt-2016,jacobs-2013,lbe-2016,
shao-2015,togni-2014, torres-2015}. In particular, it was shown that the problem of determining the packing chromatic number is computationally (very) hard~\cite{fiala-2010} as its decision version is NP-complete even when restricted to trees.  Already in the seminal paper~\cite{goddard-2008} it was observed that there is no upper bound for the packing chromatic number in the class of graphs with fixed maximum degree $\Delta$ when $\Delta\ge 4$, while the question for subcubic graphs (i.e., the graphs with $\Delta\le 3$) intrigued several authors, see~\cite{goddard-2008,bkr-2016,bkrw-2017a,gt-2016}. In particular, a subcubic graph with packing chromatic number 13 was found in \cite{gt-2016}, and a subcubic graph with packing chromatic number 14 was constructed in~\cite{bkrw-2017a}, but no subcubic graph with bigger packing chromatic number was known. Finally, in~\cite{balogh-2018} the authors proved that the packing chromatic number of subcubic graphs is unbounded. The proof is rather involved and uses the so-called configuration model technique. However, this remarkable proof does not give an explicit construction of a family of subcubic graphs with unbounded packing chromatic number. 

In this note we present a family of subcubic graphs $G_k$ with the property that $\pch(G_k)\ge 2k+9$. The main tool in the proof is to keep the diameter of the graphs in the family under control (i.e., $\diam(G_k)\le 2k+6$), and at the same time a packing coloring of these graphs requires more colors than the diameter. We are able to compute the bounds for the packing chromatic numbers of the graphs $G_k$ by using recursive structure of the family $G_k$ (each graph $G_k$ contains two copies of $G_{k-1}$ as induced subgraphs). 

In the remainder of this note, we present the construction and prove the mentioned bounds for the diameter and the packing chromatic number of the graphs $G_k$.
 The basic building block in the construction is the graph $H$ in Fig.~\ref{fig:graphH}.
Note that $\diam(H)=4$. 

%%%%%%%%%%%%%%%%%%%%%%%%%%%%%%%%%%%%%%%%%%%%%%%%%%%%%%%%%%%%%%%%%%%%%%%%%%%%%%%%%%%%%%%%%%%%%%%%%%%%%%%%%%%%%%%%%%%%%%%%%%%%%%%%%%%%%%%
%%%%%%%%%%%%%%%%%%%%%%%%%%%%%%%%%%%%% GRAF H %%%%%%%%%%%%%%%%%%%%%%%%%%%%%%%%%%%%%%%%%%%%%%%%%%%%%%%%%%%%%%%%%%%%%%%%%%%%%%%%%%%%%%%%%%
\begin{figure}[h]
\begin{center}
\begin{tikzpicture}%[scale=0.95, style=thick]
\def\vr{3pt}
\def\len{1}

\foreach \i in {1, 3}{
\coordinate(y_\i) at (\i-1, 2);
\coordinate(z_\i) at (\i+3, 2);}
\foreach \i in {5, 7}{
\coordinate(y_\i) at (\i-5, 0);
\coordinate(z_\i) at (\i-1, 0);}
\foreach \i in {2, 4, 6}{
\coordinate(y_\i) at (1, 2-\i*0.25);
\coordinate(z_\i) at (5, 2-\i*0.25);}
\coordinate (w) at (3, 1);
\draw (y_1)--(y_5)--(y_6)--(y_7)--(y_3)--(y_2)--(y_1);
\draw (y_2)--(y_4)--(y_6);
\draw (z_1)--(z_5)--(z_6)--(z_7)--(z_3)--(z_2)--(z_1);
\draw (z_2)--(z_4)--(z_6);
\draw (y_4)--(w)--(z_4);
\draw (y_3)--(z_1);
\draw (y_5) .. controls (2, -1) and (4, -1) .. (z_7);
\foreach \i in {1, 2, 3, 4, 5, 6, 7}{
\draw(y_\i)[fill=white] circle(\vr);
\draw(z_\i)[fill=white] circle(\vr);
}
\draw(w)[fill=white] circle(\vr);
\draw(y_1)node[left]{$y_1$}; \draw(y_2)node[above]{$y_2$}; \draw(y_3)node[above]{$y_3$}; \draw(y_4)node[left]{$y_4$}; \draw(y_5)node[left]{$y_5$}; \draw(y_6)node[below]{$y_6$}; \draw(y_7)node[right]{$y_7$};
\draw(z_1)node[above]{$z_1$}; \draw(z_2)node[above]{$z_2$}; \draw(z_3)node[right]{$z_3$}; \draw(z_4)node[right]{$z_4$}; \draw(z_5)node[left]{$z_5$}; \draw(z_6)node[below]{$z_6$}; \draw(z_7)node[right]{$z_7$};
\draw(w)node[above]{$w$};

\end{tikzpicture}
\end{center}
\caption{Graph $H$} 
\label{fig:graphH}
\end{figure}

%%%%%%%%%%%%%%%%%%%%%%%%%%%%%%%%%%%%%%%%%%%%%%%%%%%%%%%%%%%%%%%%%%%%%%%%%%%%%%%%%%%%%%%%%%%%%%%%%%%%%%%%%%%%%%%%%%%%%%%%%%%%%%%%%%%%%%%
%%%%%%%%%%%%%%%%%%%%%%%%%%%%%%%%%%%%%%%%%%%%%%%%%%%%%%%%%%%%%%%% PAK. KROMATIČNO ŠT. ZA H %%%%%%%%%%%%%%%%%%%%%%%%%%%%%%%%%%%%%%%%%%%%%

\begin{lemma}
The packing chromatic number of the graph $H$, shown in Fig. \ref{fig:graphH}, is at least $7$. 
\label{lema_pakirno_H}
\end{lemma}

\begin{proof}
Suppose to the contrary that $\chi_\rho(H) \leq 6$ and let $c$ be an arbitrary $6$-packing coloring of a graph $H$. Denote by $Y$ the subgraph of $H$ induced by the vertices $y_1, \ldots, y_7$, and by $Z$ the subgraph of $H$ induced by the vertices $z_1,  \ldots, z_7$.
For $c$ restricted to $Y$ (or analogously to $Z$) we have $|c^{-1}(1) \cap V(Y)| \leq 3$, $|c^{-1}(2) \cap V(Y)| \leq 2$ and $|c^{-1}(3) \cap V(Y)| \leq 1$. Therefore $|c^{-1}(1) \cap V(H)| \leq 7$, $|c^{-1}(2) \cap V(H)| \leq 5$, $|c^{-1}(3) \cap V(H)| \leq 2$. Since the diameter of $H$ is $4$, we also have $|c^{-1}(l) \cap V(H)| = 1$ for any $l \in \{4, 5, 6\}$. 
We distinguish four cases with respect to $|c^{-1}(2) \cap V(H)|$.

\textbf{Case 1.} $|c^{-1}(2) \cap V(H)| = 5$. 

All vertices, which get color $2$, are uniquely determined and these are $y_1, y_7, w, z_3$ and $z_5$. Since $c(w) \neq 1$,  we have $|c^{-1}(1) \cap V(H)| \leq 6$. Recall that $|c^{-1}(1) \cap V(Y)| \leq 3$ and $|c^{-1}(1) \cap V(Z)| \leq 3$, but it is easy to see that if $|c^{-1}(1) \cap V(Y)|=3 $ then $|c^{-1}(1) \cap V(Z)|$ cannot reach the established upper bound (respectively, if $|c^{-1}(1) \cap V(Z)|=3 $, then $|c^{-1}(1) \cap V(Y)| < 3 $). 
This yields $|c^{-1}(1) \cap V(H)| \leq 5$. If $|c^{-1}(1) \cap V(H)| = 5$ (and $|c^{-1}(2) \cap V(H)| = 5$), then vertices colored by $1$ are $y_3, y_4, y_5, z_2$ and $z_6$ (or $z_1, z_4, z_7, y_2$ and $y_6$). But then $|c^{-1}(3) \cap V(H)| \leq 1$ and hence $\sum_{i=1}^{6}|c^{-1}(i)\cap V(H)| \leq 14$. This is contradiction since $H$ has $15$ vertices, but we can color only $14$ of them. The same contradiction arises if $|c^{-1}(1) \cap V(H)| < 5$.

\textbf{Case 2.} $|c^{-1}(2) \cap V(H)| = 4$. 

Suppose that $c(w) \neq 2$. Then the vertices, which are colored by $2$, are $y_1, y_7, z_3$ and $z_5$. 
If $c(w)=1$, then $|c^{-1}(1) \cap V(H)| \leq 5$ and $\sum_{i=1}^{6}|c^{-1}(i)\cap V(H)| \leq 14$, therefore we get the same contradiction as above. If $c(w) \neq 1$, we have the analogous situation as in the Case $1$, which implies $|c^{-1}(1) \cap V(H)| \leq 5$ and $\sum_{i=1}^{6}|c^{-1}(i)\cap V(H)| \leq 14$, a contradiction. 

Next, suppose that $c(w) = 2$. Without loss of generality we may assume that $|c^{-1}(2) \cap V(Y)| = 2$ (and $|c^{-1}(2) \cap V(Z)| =1$). This is not possible if $c(y_3)=2$ and $c(y_5)=2$, hence the vertices of $Y$ colored by $2$ are $y_1$ and $y_7$. It is clear that $|c^{-1}(1) \cap V(H)| \leq 6$, but it is also easy to see that if $|c^{-1}(1) \cap V(Y)| =3$, then $|c^{-1}(1) \cap V(Z)| \leq 2$ (note that $|c^{-1}(2) \cap V(Z)| =1$, namely $c(z_3)=2$ or $c(z_5)=2$). Hence $|c^{-1}(1) \cap V(H)| \leq 5$ and we get the same contradiction as above.

\textbf{Case 3.} $|c^{-1}(2) \cap V(H)| = 3 $. 

If $c(w) \neq 1$ then $|c^{-1}(1) \cap V(H)| \leq 6$ and $\sum_{i=1}^{6}|c^{-1}(i) \cap V(H)| \leq 14$, a contradiction.

Suppose that $c(w)=1$. Again, without loss of generality we may assume that $|c^{-1}(2) \cap V(Y)| = 2$ (and $|c^{-1}(2) \cap V(Z)| =1$). The vertices of $Y$, colored by $2$, are either $y_1$ and $y_7$ or $y_3$ and $y_5$, but in each case this yields $|c^{-1}(1) \cap V(Y)| \leq 2$ (note that $c(y_4) \neq 1$). Hence $|c^{-1}(1) \cap V(H)| \leq 6$ and $\sum_{i=1}^{6}|c^{-1}(i) \cap V(H)| \leq 14$, which is a contradiction.

\textbf{Case 4.} $|c^{-1}(2) \cap V(H)| \leq 2 $. 

In this case we have $\sum_{i=1}^{6}|c^{-1}(i) \cap V(H)| \leq 14$, which is again a contradiction.
\qed
\end{proof}

%%%%%%%%%%%%%%%%%%%%%%%%%%%%%%%%%%%%%%%%%%%%%%%%%%%%%%%%%%%%%%%%%%%%%%%%%%%%%%%%%%%%%%%%%%%%%%%%%%%%%%%%%%%%%%%%%%%%%%%%%%%%%%%%%%%%%%%%%%%%%%%%%%%%%%%%%%%%%%%%%%%%%%%%%%%%%%%%%%%%%%%%%%%%%% G_0 SLIKA %%%%%%%%%%%%%%%%%%%%%%%%%%%%%%%%%%%%%%%%%%%%%%%%%%%%%%%%%%%%%%%%%%%%%%%%%%%%%%%%%%%%%%%%%%%%%%%%%%%%%%%%%%%%%%%%%%%%%%%%%%%%%%%

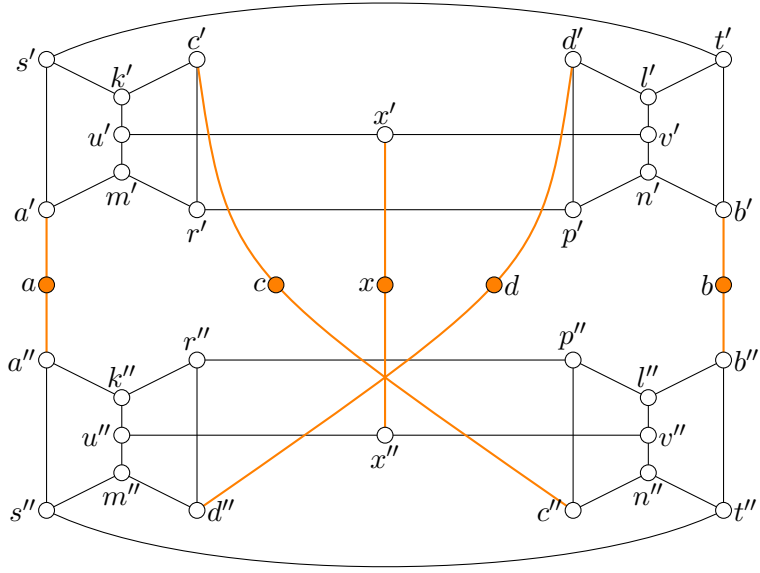
\begin{figure}[h]
\begin{center}
\begin{tikzpicture}%[scale=0.95, style=thick]
\def\vr{3pt}
\def\len{1}

\foreach \j in {2, 6}{
\foreach \i in {1, 3}{
\coordinate(y_\i^\j) at (\i-1, \j);
\coordinate(z_\i^\j) at (\i+6, \j);}}
\foreach \j in {0, 4}{
\foreach \i in {5, 7}{
\coordinate(y_\i^\j) at (\i-5, \j);
\coordinate(z_\i^\j) at (\i+2, \j);}}
\foreach \i in {2, 4, 6}{
\coordinate(y_\i) at (1, 2-\i*0.25);
\coordinate(z_\i) at (8, 2-\i*0.25);
\coordinate(y'_\i) at (1, 6-\i*0.25);
\coordinate(z'_\i) at (8, 6-\i*0.25);
}
\coordinate (w) at (4.5, 1);
\coordinate (w') at (4.5, 5);
\draw (y_1^2)--(y_5^0)--(y_6)--(y_7^0)--(y_3^2)--(y_2)--(y_1^2); 
\draw (y_2)--(y_4)--(y_6); 
\draw (z_1^2)--(z_5^0)--(z_6)--(z_7^0)--(z_3^2)--(z_2)--(z_1^2);
\draw (z_2)--(z_4)--(z_6);
\draw (y_4)--(w)--(z_4);
\draw (y_3^2)--(z_1^2);
\draw (y_5^0) .. controls (2, -1) and (7, -1) .. (z_7^0);
\draw (y_1^6)--(y_5^4)--(y'_6)--(y_7^4)--(y_3^6)--(y'_2)--(y_1^6); 
\draw (y'_2)--(y'_4)--(y'_6); 
\draw (z_1^6)--(z_5^4)--(z'_6)--(z_7^4)--(z_3^6)--(z'_2)--(z_1^6);
\draw (z'_2)--(z'_4)--(z'_6);
\draw (y'_4)--(w')--(z'_4);
\draw (y_7^4)--(z_5^4);
\draw (y_1^6) .. controls (2, 7) and (7, 7) .. (z_3^6);
%
% VMESNE
\draw[thick, orange](w)--(w');
\draw[thick, orange](y_1^2)--(y_5^4);
\draw[thick, orange](z_3^2)--(z_7^4);
\draw[thick, orange](y_3^6).. controls (2.4, 3.2) .. (z_5^0);
\draw[thick, orange](y_7^0).. controls (6.6, 3.2) ..(z_1^6);
\foreach \i in {2, 4, 6}{
\draw(y_\i)[fill=white] circle(\vr);
\draw(z_\i)[fill=white] circle(\vr);
\draw(y'_\i)[fill=white] circle(\vr);
\draw(z'_\i)[fill=white] circle(\vr);
}
\foreach \j in {2, 6}{
\foreach \i in {1, 3}{
\draw(y_\i^\j)[fill=white]circle(\vr);
\draw(z_\i^\j)[fill=white]circle(\vr);}}
\foreach \j in {0, 4}{
\foreach \i in {5, 7}{
\draw(y_\i^\j)[fill=white]circle(\vr);
\draw(z_\i^\j)[fill=white]circle(\vr);}}
\draw(w)[fill=white] circle(\vr);
\draw(w')[fill=white] circle(\vr);
\draw(y_1^2)node[left]{$a''$}; \draw(y_2)node[above]{$k''$}; \draw(y_3^2)node[above]{$r''$}; \draw(y_4)node[left]{$u''$}; \draw(y_5^0)node[left]{$s''$}; \draw(y_6)node[below]{$m''$}; \draw(y_7^0)node[right]{$d''$};
\draw(y_1^6)node[left]{$s'$}; \draw(y'_2)node[above]{$k'$}; \draw(y_3^6)node[above]{$c'$}; \draw(y'_4)node[left]{$u'$}; \draw(y_5^4)node[left]{$a'$}; \draw(y'_6)node[below]{$m'$}; \draw(y_7^4)node[below]{$r'$};
\draw(z_1^2)node[above]{$p''$}; \draw(z_2)node[above]{$l''$}; \draw(z_3^2)node[right]{$b''$}; \draw(z_4)node[right]{$v''$}; \draw(z_5^0)node[left]{$c''$}; \draw(z_6)node[below]{$n''$}; \draw(z_7^0)node[right]{$t''$};
\draw(z_1^6)node[above]{$d'$}; \draw(z'_2)node[above]{$l'$}; \draw(z_3^6)node[above]{$t'$}; \draw(z'_4)node[right]{$v'$}; \draw(z_5^4)node[below]{$p'$}; \draw(z'_6)node[below]{$n'$}; \draw(z_7^4)node[right]{$b'$};
\draw(w)node[below]{$x''$};
\draw(w')node[above]{$x'$};
\coordinate(x) at (4.5, 3); \draw(x)[fill=orange]circle(\vr); \draw(x)node[left]{$x$};
\coordinate(a) at (0, 3); \draw(a)[fill=orange]circle(\vr); \draw(a)node[left]{$a$};
\coordinate(b) at (9, 3); \draw(b)[fill=orange]circle(\vr); \draw(b)node[left]{$b$};
\coordinate(c) at (3.05, 3); \draw(c)[fill=orange]circle(\vr); \draw(c)node[left]{$c$};
\coordinate(d) at (5.95, 3); \draw(d)[fill=orange]circle(\vr); \draw(d)node[right]{$d$};
\end{tikzpicture}
\end{center}
\caption{Graph $G_0$} 
\label{fig:G1}
\end{figure}

%%%%%%%%%%%%%%%%%%%%%%%%%%%%%%%%%%%%%%%%%%%%%%%%%%%%%%%%%%%%%%%%%%%%%%%%%%%%%%%%%%%%%%%%%%%%%%%%%%%%%%%%%%%%%%%%%%%%%%%%%%%%%%%%%%%%%%%%%%%%%%%%%%%%%%%%%%%%%%%%%%%%%%%%%%%%%%%%%%%%%%%%%%%%%%%%%%%%%%%%%%%%%%%%%%%%%%%%%%%%%%%%% G_1 %%%%%%%%%%%%%%%%%%%%%%%%%%%%%%%%%%%%%%%%%%

Next, starting from two copies of graph $H$ (denoted by $H'$ and $H''$), adding five vertices, and adding edges as shown in Fig.~\ref{fig:G1} we obtain graph $G_0$. 

%%%%%%%%%%%%%%%%%%%%%%%%%%%%%%%%%%%%%%%%%%%%%%%%%%%%%%%%%%%%%%%%%%%%%%%%%%%%%%%%%%%%%%%%%%%%%%%%%%%%%%%%%%%%%%%%%%%%%%%%%%%%%%%%%%%%%%%%%%%%%%%%%%%%%%%%%%%%%%%%%%%%%%%%%%%%%%%%%%%%%%%%%%%%%%%%%%%%%%%%%%%%%%%%%%%%%%%%%%%%%%%%%%%%%%% DIAMETER G_0 %%%%%%%%%%%%%%%%%%%%%%%%%%%
\begin{lemma}
\label{lem:diam}
The diameter of the graph $G_0$, shown in Fig. \ref{fig:G1}, is at most $6$. 
\end{lemma}

\begin{proof}
Recall that the diameter of the graph $H$ is $4$. Hence it is clear that the distance between any two vertices of the subgraph $H'$ (or $H''$) of $G_0$ is at most $4$. Therefore we only need to check the distances between any two vertices of $G_0$, of which one is from $V(H')\cup \{a, b, c, d, x\}$ and the other from $V(H'') \cup \{a, b, c, d, x\}$. 

Each vertex from $\{a, b, c, d, x\}$ is adjacent to some vertex of $H'$ and also to some vertex of $H''$. Since the diameters of $H'$ and $H''$ are $4$, the distance between any vertex from $\{a, b, c, d, x\}$ and any other vertex of $G_0$ is at most $6$. 
 
Next, consider vertices $a', b', c'$, $d'$ and $x'$. Each of them is adjacent to some vertex from $\{a, b, c, d, x\}$ and hence is at distance $2$ from some vertex of $H''$. Therefore it is at distance at most $6$ from any vertex of $H''$ (actually from any vertex of $G_0$). By symmetry also vertices $a'', b'', c''$, $d''$ and $x''$ are at distance at most $6$ from any other vertex of $G_0$. 

Any vertex from $\{p', s', r', t'\}$ is at distance at most $3$ from vertices $a$ and $b$ or from vertices $c$ and $d$. This yields that any mentioned vertex is at distance at most $6$ from any vertex of $G_0$. By symmetry the same holds also for vertices $p'', s'', r''$ and $t''$. 

Note that the distance between $x$ and any vertex of $H'$ or $H''$ is $4$. Since the vertices $u', v', u''$ and $v''$ are at distance $2$ from vertex $x$, it is clear, that these vertices are at distance at most $6$ from any vertex of $G_0$. 

We still need to check the distances between vertices $k', l', m', n', k'', l'', m''$ and $n''$. As mentioned above, the distance between any two of them is at most $4$, if both of them belong either to $H'$ or to $H''$. Otherwise, for any two mentioned vertices there exist a path of length $6$ through vertex $x$, and hence any such two vertices are at distance at most $6$ in $G_0$. 
\qed
\end{proof}

%%%%%%%%%%%%%%%%%%%%%%%%%%%%%%%%%%%%%%%%%%%%%%%%%%%%%%%%%%%%%%%%%%%%%%%%%%%%%%%%%%%%%%%%%%%%%%%%%%%%%%%%%%%%%%%%%%%%%%%%%%%%%%%%%%%%%%%%%%%%%%%%%%%%%%%%%%%%%%%%%%%%%%%%%%%%%%%%%%%%%%%%%%%%%% PACKING CHROMATIC NUMBER OF G_0 %%%%%%%%%%%%%%%%%%%%%%%%%%%%%%%%%%%%%%%%%%%%%%%%%
\begin{lemma}
\label{lem:pakG0}
The packing chromatic number of the graph $G_0$ is at least $9$. 
\end{lemma}

\begin{proof}
Recall that $G_0$ consists of two distinct copies of subgraphs isomorphic to $H$. By Lemma \ref{lema_pakirno_H} for a packing coloring of each of the two copies at least $7$ colors is required. But since the diameter of $G_0$ is at most $6$, the colors $6$ and $7$ can be used only in one copy of the graph $H$, so in the other copy of $H$ they need to be substituted by (two) additional colors. Therefore for a packing coloring of the entire graph $G_0$ at least $9$ colors is required (actually $9$ colors is already required for a packing coloring of vertices in $V(H') \cup V(H'')$). 
\qed
\end{proof}

%%%%%%%%%%%%%%%%%%%%%%%%%%%%%%%%%%%%%%%%%%%%%%%%%%%%%%%%%%%%%%%%%%%%%%%%%%%%%%%%%%%%%%%%%%%%%%%%%%%%%%%%%%%%%%%%%%%%%%%%%%%%%%%%%%%%%%%%%%%%%%%%%%%%%%%%%%%%%%%%%%%%%%%%%%%%%%%%% LEMA: RAZDALJE V GRAFU %%%%%%%%%%%%%%%%%%%%%%%%%%%%%%%%%%%%%%%%%%%%%%%%%%%%%%%%%%%%%%%%%%%%%%%
\begin{lemma}
For any two (not necessarily distinct) vertices $z,w$ of $G_0$ there exists a vertex $y_{z,w}$ in $G_0$ with $\deg(y_{z,w})=2$ such that $d(z,y_{z,w})+d(y_{z,w},w)\le 6$.
\label{lem:stopnje}
\end{lemma}

\begin{proof}
Let $z$ and $w$ be any two vertices in $G_0$. If at least one of them has degree $2$, the statement follows from Lemma~\ref{lem:diam}. 

Suppose that one of the vertices $z$ or $w$ belongs to $V(H')$ and the other to $V(H'')$. Then each $z,w$-path contains a vertex of degree $2$, denote it by $y_{z,w}$. Clearly, this is also true for a shortest $z,w$-path, of which length is at most $6$, since the diameter of $G_0$ is at most $6$. 
Therefore the sum of the distances from vertices $z$ and $w$ to vertex $y_{z,w}$ (of degree $2$) is at most $6$.

Next, without loss of generality suppose that both, $z$ and $w$, belong to $V(H')$. 
Since the diameter of the subgraph $H'$ of $G_0$ is $4$, it is clear that the distance between $z$ (or $w$) and any vertex in $\{a, b, c, d, x\}$ is at most $5$. Therefore, if $w$ (or $z$) belongs to $\{a', b', c', d', x'\}$, then the lemma holds (in this case $y_{z,w}$ is a neighbour of $w$). 

The lemma also holds if $z$ (or $w$) is at distance $2$ from $x$ or if $z, w \in \{k', m', l', n'\}$. In both cases the vertex $y_{z,w}$ is provided by $x$; in the first case we use the fact that the distance between $x$ and any vertex from $V(H')$ is at most $4$ in $G_0$, and in the second case we use the fact that $z$ and $w$ are both at distance $3$ from $x$.

\begin{table}[!ht]
\centering
\caption{Vertices $z, w$ and $y_{z,w}$ from the proof of Lemma \ref{lem:stopnje}}
\label{tabela1}
\begin{tabular}{|l|l|l|l|}
 \hline $z$ & $w$ & $y_{z,w}$ & $d(z, y_{z,w})+d(w, y_{z,w})$  \\
 \hline $s'$ & $r'$ & $a$ & $2+3=5$  \\
 \hline $s'$ & $n'$ & $b$ & $3+2=5$   \\
 \hline $s'$ & $p'$ & $a$ & $2+4=6$   \\
 \hline $k'$ & $p'$ & $c$ & $2+3=5$   \\
 \hline $m'$ & $t'$ & $a$ & $2+3=5$   \\
 \hline $r'$ & $t'$ & $a$ & $3+3=6$   \\
 \hline $r'$ & $l'$ & $d$ & $3+2=5$   \\
 \hline $p'$ & $t'$ & $b$ & $3+2=5$   \\
\hline 
\end{tabular}
\end{table}

We still need to check some pairs of vertices in $\{k', l', m', n', p', r', s', t'\}$. Each listed vertex is at distance $2$ from some vertex of $\{a, b, c, d\}$. Hence in the case when the distance between two listed vertices is at most $2$, the lemma holds, namely $y_{z,w}$ is provided by the vertex, which is at distance $2$ from $z$ (resp., $w$), $w$ (resp., $z$) is then at distance at most $4$ from $y_{z,w})$. 
The pairs of vertices of degree $3$, which we actually need to check are written in Table \ref{tabela1}. For each pair the corresponding vertex of degree $2$ is also listed, which provides that the corresponding sum of the distances is at most $6$.

In the case when $z$ coincides with $w$, the statement clearly holds, since each vertex of $H'$ (resp. $H''$) is at distance at most $2$ from some vertex in $\{a, b, c, d, x\}$. 

\qed
\end{proof}

%%%%%%%%%%%%%%%%%%%%%%%%%%%%%%%%%%%%%%%%%%%%%%%%%%%%%%%%%%%%%%%%%%%%%%%%%%%%%%%%%%%%%%%%%%%%%%%%%%%%%%%%%%%%%%%%%%%%%%%%%%%%%%%%%%%%%%%%%%%%%%%%%%%%%%%%%%%%%%%%%%%% KONSTRUKCIJA %%%%%%%%%%%%%%%%%%%%%%%%%%%%%%%%%%%%%%%%%%%%%%%%%%%%%%%%%%%%%%%%%%%%%%%%%%%%%%%%%%%%%%%%%%%%%%

We continue by presenting the family of graphs $G_k$, which possess the desired properties. Recall that a {\em perfect binary tree} is a (rooted) binary tree in which all interior vertices have two children and all leaves have the same depth. (The orientation of the tree is used only for the reason of easier presentation, yet the resulting tree is considered as non-oriented.)
Next, we present a natural labelling of vertices in a perfect binary tree $T$. Firstly, the root is denoted by the empty label, while given a label $\ell$ of an interior vertex in the tree, the labels of its two children are obtained from $\ell$ by adding a bit to the left-hand side of the label $\ell$; more precisely, $0$ is added to the left of $\ell$ for the left child and $1$ is added to the left of $\ell$ for the right child. In this way, vertices in the $m$th level of $T$ (having distance $m$ from the root) have as its label an $m$-tuple, which consists of binary values (zeros and ones). In particular, the left-most leaf in $T$ is denoted by $0\ldots 0$, while the right-most leaf by $1\ldots 1$, where the number of zeros (resp., ones) coincides with the depth of $T$. 
To distinguish vertices of $T$ by vertices of other perfect binary trees, we denote its vertices by $T(\beta_1 \ldots \beta_m)$, where $\beta_i\in\{0,1\}$ for all $i$; see Fig.~\ref{fig:tree}.

%%%%%%%%%%%%%%%%%%%%%%%%%%%%%%%%%%%%%%%%%%%%%%%%%% SLIKA KONST. %%%%%%%%%%%%%%%%%%%%%%%%%%%%%%%%%%%%%%%%%%%%%%%%%%%%%%%%%%%%%%%%%%%%%%%%
\begin{figure}[h]
\begin{center}
\begin{tikzpicture}%[scale=0.95, style=thick]
\def\vr{3pt}
\def\len{1}

\coordinate(y_0^0) at (0,0);

\foreach \i in {2, 4, 6, 8, 10, 12, 14}{
\ifthenelse{\i < 7}{\coordinate(y_\i^0) at (\i-0.1*\i, 0);}{\coordinate(y_\i^0) at (\i-1-0.1*\i, 0);}
}
\foreach \i in {1, 5, 9, 13}{
\ifthenelse{\i < 7}{\coordinate(y_\i^2) at (\i-0.1*\i, 2);}{\coordinate(y_\i^2) at (\i-1-0.1*\i, 2);}
} 
\coordinate (y_3^4) at (2.7, 4);
\coordinate (y_11^4) at (9, 4);
\coordinate (y_7^6) at (5.8, 6);  
\draw (y_0^0)--(y_1^2)--(y_3^4)--(y_7^6)--(y_11^4)--(y_13^2)--(y_14^0); \draw (y_1^2)--(y_2^0); \draw (y_3^4)--(y_5^2)--(y_6^0); \draw (y_4^0)--(y_5^2); \draw (y_13^2)--(y_12^0); \draw (y_11^4)--(y_9^2)--(y_8^0); \draw (y_10^0)--(y_9^2);
\foreach \i in {0, 2, 4, 6, 8, 10, 12, 14}{
\draw(y_\i^0)[fill=white] circle(\vr);}
\foreach \i in {1, 5, 9, 13}{
\draw(y_\i^2)[fill=white] circle(\vr); }
\foreach \i in {3, 11}{ 
\draw(y_\i^4)[fill=white] circle(\vr);  
}
\draw(y_7^6)[fill=white] circle(\vr);  
\draw(y_7^6)node[above]{\small{$T()$}};
\draw(y_3^4)node[left]{\small{$T(0)$}};
\draw(y_11^4)node[right]{\small{$T(1)$}};
\draw(y_1^2)node[left]{\small{$T(00)$}};
\draw(y_5^2)node[left]{\small{$T(10)$}};
\draw(y_9^2)node[right]{\small{$T(01)$}};
\draw(y_13^2)node[right]{\small{$T(11)$}};
\draw(y_0^0)node[left]{\footnotesize{$T(000)$}};
\draw(y_2^0)node[left]{\footnotesize{$T(100)$}};
\draw(y_4^0)node[left]{\footnotesize{$T(010)$}};
\draw(y_6^0)node[left]{\footnotesize{$T(110)$}};
\draw(y_8^0)node[right]{\footnotesize{$T(001)$}};
\draw(y_10^0)node[right]{\footnotesize{$T(101)$}};
\draw(y_12^0)node[right]{\footnotesize{$T(011)$}};
\draw(y_14^0)node[right]{\footnotesize{$T(111)$}};

\end{tikzpicture}
\end{center}
\caption{Perfect binary tree $T$ of depth $3$ and the described labelling of its vertices} 
\label{fig:tree}
\end{figure}
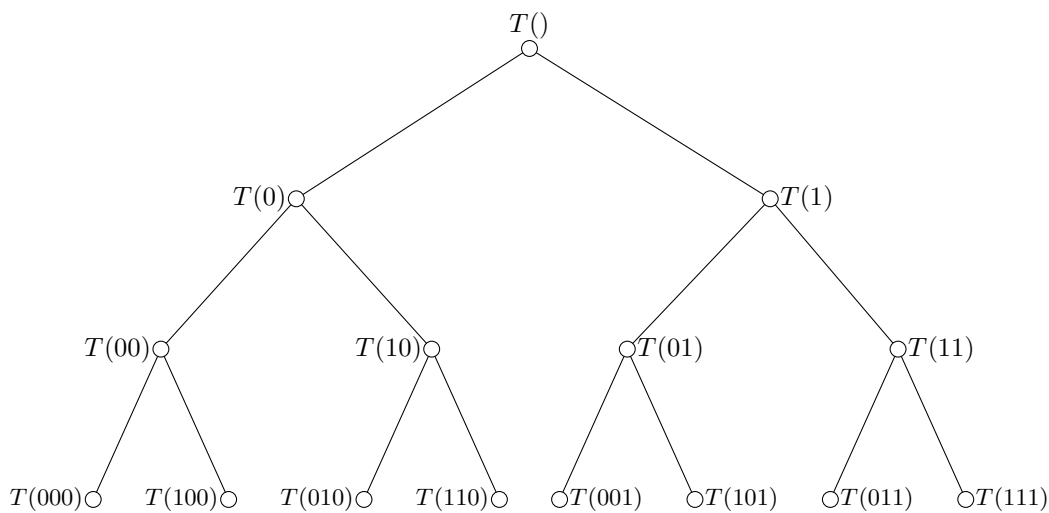

For any positive integer $k$, we begin the construction of the graph $G_k$ by taking $5$ copies of a perfect binary tree of depth $k$, (thus) each having $2^k$ leaves. The trees are denoted by $A,B,C,D$ and $X$ (suggesting to which vertices in the graphs $G_0$ they will be attached), and their vertices are labelled as described above. 
Now, add $2^k$ copies of the graph $G_0$ and attach them to the existing five trees as follows. For each binary $k$-tuple $\beta_1\ldots\beta_k$, where $\beta_i\in\{0,1\}$ for all $i$, take a copy of $G_0$, and denote it by $G_0(\beta_1 \ldots \beta_k)$. Now, identify the vertex $a\in V(G_0(\beta_1\ldots\beta_k))$ with $A(\beta_1 \ldots \beta_k)$, identify the vertex $b\in V(G_0(\beta_1 \ldots \beta_k))$ with $B(\beta_1 \ldots \beta_k)$, identify the vertex $c\in V(G_0(\beta_1 \ldots \beta_k))$ with $C(\beta_1 \ldots \beta_k)$, identify the vertex $d\in V(G_0(\beta_1 \ldots \beta_k))$ with $D(\beta_1 \ldots \beta_k)$, and identify the vertex $x\in V(G_0(\beta_1 \ldots \beta_k))$ with $X(\beta_1 \ldots \beta_k)$. Note that vertices of a copy of $G_0$ are identified only with leaves of (distinct) perfect binary trees. 

In particular, $G_1$ is obtained from two copies of $G_0$, namely $G_0(0)$ and $G_0(1)$, by adding edges between $a$ (resp., $b,c,d,x$) in $V(G_0(0))$ to $a$ (resp., $b,c,d,x$) in $V(G_0(1))$, and then subdividing these five new edges.

%%%%%%%%%%%%%%%%%%%%%%%%%%%%%%%%%%%%%%%%%%%%%%%%%%%%%%%%%%%%%%%%%%%%%%%%%%%%%%%%%%%%%%%%%%%%%%%%%%%%%%%%%%%%%%%%%%%%%%%%%%%%%%%%%%%%%%%%%%%%%%%%%%%%%%%%%%%%%%%%%%%%%%%%%%%%%%%%% LEMA: G_1 %%%%%%%%%%%%%%%%%%%%%%%%%%%%%%%%%%%%%%%%%%%%%%%%%%%%%%%%%%%%%%%%%%%%%%%

\begin{lemma} For the graph $G_1$ we have
$\diam(G_1)\le 8$, and $\pch(G_1)\ge 11$. 
\label{lem:G1}
\end{lemma}
\begin{proof}
To prove the bound on the diameter of $G_1$ we distinguish several cases. Firstly, if two vertices $z,w$ are from the same copy of $G_0$, then clearly, $d_{G_1}(z,w)\le 6$ by Lemma~\ref{lem:diam}. Next, if $z$ and $w$ are in distinct copies of $G_0$ (say $z\in V(G_0(0)), w\in V(G_0(1)))$, then by Lemma~\ref{lem:stopnje} there exists a vertex in each of the copies, which is one of the vertices in $\{a,b,c,d,x\}$, such that the sum of the distances from $y$ to the vertex of its copy of $G_0$ and $z$ to the vertex of its copy of $G_0$ is bounded by $6$. More precisely, there exists a vertex $y_{z,w}\in V(G_0(0))$ and a vertex $y'_{z,w}\in V(G_0(1))$ that belong to the same binary tree ($A,B,C,D,$ or $X$) such that $d_{G_1}(z,y_{z,w})+d_{G_1}(w,y'_{z,w})\le 6$. Since $y_{z,w}$ and $y'_{z,w}$ belong to the same binary tree, their distance is 2, hence $d_{G_1}(z,w)\le d_{G_1}(z,y_{z,w})+d_{G_1}(y_{z,w}, y'_{z,w})+d_{G_1}(y'_{z,w},w)\le 8$. The third case is that $w$ and $z$ are roots of two distinct binary trees. In this case, each of them is at distance $1$ from some vertex in $G_0(0)$, and by using $\diam(G_0)\le 6$ we infer that $d_{G_1}(y,z)\le 8$. Finally, if only $z$ (resp., $w$) is the root of some binary tree, then clearly $d_{G_1}(y,z)\le 7$.
This concludes the proof of the claim that $\diam(G_1)\le 8$.

As $G_1$ contains two distinct copies of (induced) subgraphs isomorphic to $G_0$, by Lemma \ref{lem:pakG0} at least $9$ colors is required for a packing coloring of each of them. But since $\diam(G_1)\le 8$, the colors $8$ and $9$ can be used in only one copy of $G_0$ in $G_1$, so in the other copy they need to be substituted by (two) additional colors. Therefore for a packing coloring of the entire graph $G_1$ at least $11$ colors are required (actually $11$ colors are required already for the packing coloring of vertices in both copies of $G_0$). 
\qed
\end{proof}

We follow with our main result.

\begin{theorem} 
\label{th:main}
For any positive integer $k$, $\diam(G_k)\le 2k+6$, and $\pch(G_k)\ge 2k+9.$ 
\end{theorem}
\begin{proof}
First, we prove that $\diam(G_k)\le 2k+6$. If $z$ and $w$ are two vertices in a copy of $G_0$, say $z\in V(G_0(\beta_1\ldots \beta_k)), w\in V(G_0(\gamma_1\ldots\gamma_k))$, then by Lemma~\ref{lem:stopnje} there exist two vertices $y_{z,w}\in V(G_0(\beta_1 \ldots \beta_k))$ and $y_{w,z}\in V(G_0(\gamma_1\ldots\gamma_k))$, where $d_{G_k}(z,y_{z,w})+d_{G_k}(w,y_{w,z})\le 6$, such that $y_{z,w}$ and $y_{w,z}$ belong to the same binary tree in $G_k$ (either $A,B,C,D$ or $X$). As $d_{G_k}(y_{z,w},y_{w,z})\le 2k$, we infer that $d_{G_k}(z,w)\le 2k+6$. 
Similarly, if only one of the vertices $z,w$, say $z$, is in a copy of $G_0$ and the other (namely, $w$) is an internal vertex of some binary tree, then one also derives that $d_{G_k}(z,w)\le 2k+6$. Indeed, a shortest path from $w$ to a vertex in the copy of $G_0$ in which $z$ lies is less than $2k$, and by using that $\diam(G_0) \leq 6$ we infer the desired bound. Finally, let $z$ and $w$ be two internal vertices of some binary tree; without loss of generality, we may assume that $z\in V(A)$ and $w\in V(B)$ (if they belong to the same binary tree, the proof is even easier). Also we may assume that the distance from $z$ to the root of $A$, vertex $A()$, equals $\ell$ and is at least as big as the distance from $w$ to the root of $B$, vertex $B()$. Hence, the label of $z$ is $A(\beta_{k-\ell+1} \ldots \beta_k)$, where $\beta_{k-i+1}\in\{0,1\}$ for all $i\in\{1,\ldots,\ell\}$.
Now, the vertex $A(0 \ldots 0 \beta_{k-\ell+1} \ldots \beta_k)$ belongs to the copy of $G_0$, namely $G_0(0 \ldots 0 \beta_{k-\ell+1} \ldots \beta_k)$, and is at distance $k-\ell$ from $z$. Clearly, $d(w,B(0 \ldots 0 \beta_{k-\ell+1} \ldots \beta_k))\le k+\ell$, hence 
\begin{equation*}
\begin{split}
d(z,w) & \le d(z,A(0 \ldots 0 \beta_{k-\ell+1} \ldots \beta_k))+ \\
&+ d(A(0 \ldots 0 \beta_{k-\ell+1} \ldots \beta_k),B(0 \ldots 0 \beta_{k-\ell+1} \ldots \beta_k))+ \\
&+d(B(0,\ldots 0 \beta_{k-\ell+1} \ldots \beta_k),w)\\
& \le (k-\ell)+6+(k+\ell) \\
& = 2k+6.
\end{split}
\end{equation*}

For the proof that $\pch(G_k)\ge 2k+9$ we use induction on $k$, and note that induction basis, $k=1$, was proven in Lemma~\ref{lem:G1}. 
For the inductive step note that each $G_k$ can be obtained from two copies of $G_{k-1}$ by adding five new vertices, and connect each of them to the two roots of the corresponding perfect binary trees.
As $G_k$ contains two distinct copies of (induced) subgraphs isomorphic to $G_{k-1}$, by induction hypothesis at least $2k+7$ colors is required for a packing coloring of each of the two copies. But since $\diam(G_k)\le 2k+6$, the colors $2k+6$ and $2k+7$ can be used in only one copy of $G_{k-1}$ in $G_k$, so in the other copy they need to be substituted by (two) additional colors. Therefore for a packing coloring of the entire graph $G_k$ at least $2k+9$ colors are required (actually $2k+9$ colors are required already for the packing coloring of vertices in the $2^{k}$ copies of $G_{0}$). 

\qed
\end{proof}

Since the graphs $G_k$ are clearly subcubic, Theorem~\ref{th:main} shows that the family $G_k$ has the property announced in the title of this note. 
Note that the graphs $G_k$ are not planar, therefore the following question still remains open.

\begin{question}
Is the packing chromatic number in the class of subcubic planar graphs bounded?
\end{question}

\section*{Acknowledgements}

B.B. acknowledges the financial support from the Slovenian Research Agency (research core funding No.\ P1-0297).

%%%%%%%%%%%%%%%%%%%%%%%%%%%%%%%%%%%%%%%%%%%%%%%%%%%%%%%%%%%%%%%%%%%%%
%%%%%%%%%%%%%%%%%%%%%%%%%%%%%%%%%%%%%%%%%%%%%%%%%%%%%%%%%%%%%%%%%%%%%

\end{document}